\documentclass[a4paper,10pt]{amsart}

\usepackage[utf8]{inputenc}

\usepackage{graphicx}

\usepackage{cite}

\usepackage[english]{babel}

\usepackage{hyperref}

\hypersetup{%
pdftitle={},
pdfsubject={Mathematics},
pdfauthor={Manfred Madritsch},
pdfkeywords={}
hyperindex=true,plainpages=false}

\usepackage{a4wide}

\usepackage{amsmath}
\usepackage{amsfonts}
\usepackage{amssymb}
\usepackage{amsthm}
\usepackage[initials]{amsrefs}

\newtheorem{lem}{Lemma}[section]
\newtheorem{thm}[lem]{Theorem}
\newtheorem{prop}[lem]{Proposition}

\numberwithin{equation}{section}

\newtheorem*{cor*}{Corollary}
\newtheorem*{thm*}{Theorem}

\theoremstyle{definition}

\theoremstyle{remark}
\newtheorem{rem}[lem]{Remark}

\allowdisplaybreaks[3]

\newcommand{\NN}{\mathbb{N}}
\newcommand{\ZZ}{\mathbb{Z}}

\newcommand{\RR}{\mathbb{R}}

\newcommand{\abs}[1]{\left| #1 \right|}
\newcommand{\norm}[1]{\left\| #1 \right\|}
\newcommand{\floor}[1]{\left\lfloor #1 \right\rfloor}

\title{Waring's problem for pseudo-polynomials}

\author[M. G. Madritsch]{Manfred G. Madritsch}
\address[M. G. Madritsch]{
\noindent Technical University of Leoben, Chair for Applied Mathematics, Leoben, Austria}
\email{manfred.madritsch@unileoben.ac.at}



\date{\today}

\begin{document}

\begin{abstract}
  Waring's problem has a long history in additive number theory. In its original
  form it deals with the representability of every positive integer as sum of
  $k$-th powers with integer $k$. Instead of these powers we deal with
  pseudo-polynomials in this paper. A pseudo-polynomial is a ``polynomial'' with
  at least one exponent not being an integer.
  
  Our work extends earlier results on the related problem of Waring for arbitrary real powers
  $k>12$ by Deshouillers and Arkhipov and Zhitkov.
\end{abstract}

\maketitle

\section{Introduction}

Inspired by Langrange's four square theorem, Edward Waring asked in a letter to
Euler, if, for a given integer $k\geq3$, does there exist a positive integer
$s\geq1$ such that every positive integer $N\geq1$ is the sum of $s$ $k$th
powers,
\textit{i.e.}
\begin{gather}\label{eq:warings-problem}
  N=n_1^k+\cdots+n_s^k.
\end{gather}
This became known as Waring's problem and was solved in 1909 by Hilbert
\cite{hilbert1909:beweis_fuer_darstellbarkeit}. We denote by $g(k)$ the smallest
$s$, such that every positive integer $N$ has a representation of the form
\eqref{eq:warings-problem}. This number has been successively decreased and
amongst other we want to mention Hardy and Littlewood
\cite{hardy_littlewood1916:contributions_to_theory}, who developed the circle
method for not only giving an upper bound for $s$ but also obtaining an
asymptotic formula for the number of representations of $N$ depending on $k$ and
$s$. Their method has been refined by Vinogradov
\cite{vinogradov1935:warings_problem} and we will use this refinement in the
present paper.

Already in the work of Hardy and Littlewood
\cite{hardy_littlewood1922:some_problems_partitio} they considered a related
quantity. We denote by $G(k)$ the smallest $s$ such that every
\textit{sufficiently large} $N$ has a representation of the form
\eqref{eq:warings-problem}. Clearly $G(k)\leq g(k)$. Since squares are congruent
to $0$, $1$ or $4\bmod 8$, one needs at least $4$ squares to represent an
integer $N\equiv 7\bmod 8$, \textit{i.e.} $G(2)=4$. In 1939 Davenport
\cite{davenport1939:warings_problem_fourth} showed that $G(4)=16$. Despite the
age of this problem not much more is known.

At almost the same time, Segal \cite{segal1934:waring_s_theorem} was able to
show a generalisation for non-integral powers. More precisely, he proved that
there exists a positive integer $s\geq1$, depending only on $c>1$, such that
\begin{gather}\label{eq:piatetski-shapiro-warings-problem}
  N=\left\lfloor n_1^c\right\rfloor
    +\cdots+
    \left\lfloor n_s^c\right\rfloor.
\end{gather}
for every sufficiently large $N$. Analogous to the above, we denote by $G(c)$
the smallest $s$, such that \eqref{eq:piatetski-shapiro-warings-problem} holds
for sufficiently large $N$. Deshouillers
\cite{deshouillers1973:probleme_de_waring} and Arkhipov and Zhitkov
\cite{arkhipov_zhitkov1984:warings_problem_with} independently gave an
asymptotic formula for the number of representation if $c>12$ and not an integer.

The aim of the present article is to on the one hand reduce this restriction to
$c>1$ and on the other hand to allow pseudo-polynomials. We call a function $f$
a pseudo-polynomial if $f$ is of the form
\begin{gather}\label{eq:pseudo-polynomial}
  f(x)=a_dx^{\theta_d}+\cdots+a_1x^{\theta_1}
  =\sum_{i=1}^d a_ix^{\theta_i}
\end{gather}
with $a_i\in\RR$ for $1\leq i\leq d$, $1\leq\theta_1<\cdots<\theta_d$ and at least
one $1\leq i\leq d$ such that $\theta_i\not\in\ZZ$.

In the present paper we consider representations of the form
\begin{gather}\label{eq:function-warings-problem}
  N=\left\lfloor f(n_1)\right\rfloor
    +\cdots+
    \left\lfloor f(n_s)\right\rfloor,
\end{gather}
where $n_i\in\NN$ for $1\leq i\leq s$ and $f$ is a pseudo-polynomial as in
\eqref{eq:pseudo-polynomial}. In the same way, we denote by $G(f)$ the smallest
$s$, such that \eqref{eq:function-warings-problem} holds for all sufficiently
large integers $N$.

However, we do not want to tackle this problem directly, but rather apply the
circle method giving us an asymptotic formula for the number of representations.
To this end, we denote by $r_{f,s}(N)$ the number of different
$(n_1,\ldots,n_s)\in\NN^s$ such that \eqref{eq:function-warings-problem} holds.
Then our main result is the following.
\begin{thm}\label{thm:main}
  Let $f$ be a pseudo-polynomial as in \eqref{eq:pseudo-polynomial} and set
  $\rho=\min(\theta_d-\theta_{d-1},\tfrac16)$ with $\theta_0=0$. Suppose
  that
  \[
    s>\frac{2}{\rho}\lceil\theta\rceil^2(\lceil\theta\rceil+1).
  \]
  Then there exists $\delta>0$ such that
  \[
    r_{f,s}(N)=\left(\frac{1}{a_d}\right)^{\frac{s}{\theta_d}}
    \frac{\Gamma\left(1+\frac{1}{\theta_d}\right)^s}
      {\Gamma\left(\frac{s+1}{\theta_d}\right)}
    N^{\frac{s}{\theta_d}-1}
  + \mathcal{O}\left(N^{\frac{s}{\theta_d}-1-\delta}\right).
  \]
\end{thm}

The rest of the paper is devoted to the proof of this result. Using the circle
method in the following section we divide the unit interval into two parts -- a
major and a minor arc. We start our consideration with the major arc, which
provides us the main term. In Section~\ref{sec:the-major-arc} we use a weight
function to rewrite the integral over the major arc into a weighted exponential
integral, which we solve in Section \ref{sec:the-singular-integral} to obtain
the main term in our asymptotic formula. Then we turn our attention to the minor
arc, which -- as we will show -- only contributes to the error term. First we
approximate the floor function using Fourier analysis in Section
\ref{sec:the-floor-function}. Then in Section
\ref{sec:exponential-sum-estimates} we give a general estimation of the
occurring exponential sums. Finally we show that the minor arc contributes to
the error term in Section \ref{sec:the-minor-arc}. In the last section we prove
Theorem \ref{thm:main} by putting together the estimate for the major and minor
arc.

\section{The setup}\label{sec:dividing-the-unit-circle}

Throughout the proof we fix a sufficiently large positive integer $N$ and a
pseudo-polynomial $f$ with $d\geq1$, coefficients $a_i$ and exponents $\theta_i$ for
$1\leq i\leq d$ as in \eqref{eq:pseudo-polynomial}. Furthermore we write $e(x):=\exp(2\pi i x)$ for short.
Then the circle method is based on the following orthogonality relation: For
$x\in\RR$ we have
\[
  \int_{0}^{1} e(\alpha x)\mathrm{d}\alpha=
  \begin{cases}1 &\text{if }x=0,\\ 0 &\text{otherwise}.\end{cases}
\]
Thus
\[
  r_{f,s}(N)=\sum_{n_1=1}^\infty\cdots\sum_{n_s=1}^\infty
    \int_{0}^{1}e\left(\alpha\left\lfloor f(n_1)\right\rfloor
      +\cdots + \alpha\left\lfloor f(n_s)\right\rfloor-\alpha N\right).
\]

In the first step we want to cut the infinite sums. To this end let $P$ be the
largest solution of $N=f(x)$. Then by the structure of $f$ we get that
\begin{equation}\label{eq:relation_P_N}
  \begin{split}
    P-\left(\frac{N}{a_d}\right)^{\frac{1}{\theta_d}}
    =P-\left(\frac{f(P)}{a_d}\right)^{\frac{1}{\theta_d}}
    =P-\left(P^{\theta_d}+\mathcal{O}\left(P^{\theta_{d-1}}\right)\right)^{\frac{1}{\theta_d}}
    \ll P^{\theta_{d-1}-\theta_d+1}.
  \end{split}
\end{equation}
Clearly, if $n>P$ for a $1\leq i\leq s$, then the integral is zero. Therefore
it suffices to consider
\begin{gather}\label{eq:waring_central_equation}
  r_{f,s}(N)
  =\sum_{n_1=1}^P\cdots\sum_{n_s=1}^P
    \int_{0}^{1}e\left(\alpha\left\lfloor f(n_1)\right\rfloor
      +\cdots + \alpha\left\lfloor f(n_s)\right\rfloor-\alpha N\right)
  =\int_{0}^{1} F(\alpha)^s e\left(-\alpha N\right)\mathrm{d}\alpha,
\end{gather}
where we have set
\begin{gather}\label{eq:F}
  F(\alpha)=\sum_{m=1}^Pe\left(\alpha \floor{f(m)}\right).
\end{gather}

In the sequel we will distinguish two cases for $\alpha\in[0,1]$. Therefore we
choose $0<v<\min\left(\theta_d-\theta_{d-1},\frac15\right)$, where we put
$\theta_0=1$ if $d=1$. Then we set
\begin{gather}\label{eq:tau}
  \tau=P^{\theta_d-v}.
\end{gather}
Now we distinguish the cases $\norm{\alpha}< \tau^{-1}$ and
$\norm{\alpha}\geq\tau^{-1}$, where $\norm{\alpha} = \min_{a\in\ZZ}\abs{\alpha-a}$
is the distance to the nearest integer. Since the integrant in
\eqref{eq:waring_central_equation} is $1$-periodic we obtain that
\begin{equation}\label{eq:circle-decomposition}
  \begin{split}
    r_{f,s}(N)
    &=\int_{-\tau^{-1}}^{1-\tau^{-1}}F(\alpha)^s e\left(-\alpha N\right)\mathrm{d}\alpha\\
    &=\int_{-\tau^{-1}}^{+\tau^{-1}}F(\alpha)^s e\left(-\alpha N\right)\mathrm{d}\alpha
      +\int_{\tau^{-1}}^{1-\tau^{-1}}F(\alpha)^s e\left(-\alpha N\right)\mathrm{d}\alpha.
  \end{split}
\end{equation}
We call
\begin{align*}
  \left]-\tau^{-1},\tau^{-1}\right[&=\left\{\alpha\in]-\tau^{-1},1-\tau^{-1}]\colon \norm{\alpha}<\tau^{-1}\right\}
  \quad\text{and}\\
  \left[\tau^{-1},1-\tau^{-1}\right]&=\left\{\alpha\in]-\tau^{-1},1-\tau^{-1}]\colon \norm{\alpha}\geq\tau^{-1}\right\}
\end{align*}
the major and minor arc, respectively.

Below we will treat the integral over these arcs differently. First for the
major arc we will use that $\alpha\left\lfloor f(m)\right\rfloor$ stays small
for $1\leq m\leq P$. Thus comparing the exponential sum $F(\alpha)$ with a
weighted one will do the job.

For the minor arc, we need more steps. First we use Fourier analysis to get rid
of the floor function. Then we estimate the occurring exponential sum. Here we
need to distinguish the cases of $\theta_d\in\ZZ$ and $\theta_d\not\in\ZZ$.
Finally we put these estimates in the Fourier series in order to obtain the
desired bounds for the minor arc.

\section{The major arc}\label{sec:the-major-arc}

For the major arc we compare the function $F(\alpha)$ with the weighted exponential sum
$V(\alpha)$ defined by
\begin{gather}\label{eq:V}
  V(\alpha)=\left(\frac{1}{a_d}\right)^{\frac{1}{\theta_d}}
    \frac{1}{\theta_d}
    \sum_{m=1}^N m^{\frac{1}{\theta_d}-1}e(\alpha m).
\end{gather}

Our first tool states that $V$ is small.
\begin{lem}\label{lem:estimate-of-V}
  If $\abs{\alpha}\leq\frac12$, then
  \[
    V(\alpha)\ll \min\left(P,\abs{\alpha}^{-\frac{1}{\theta_d}}\right)
  \]
\end{lem}

\begin{proof}
  Let
  \[
    h(x)=\left(\frac{1}{a_d}\right)^{\frac{1}{\theta_d}}
      \frac{1}{\theta_d}x^{\frac{1}{\theta_d}-1}.
  \]
  Then we obtain for the absolute value of $V(\alpha)$
  \begin{align*}
    \abs{V(\alpha)}
    \leq \left(\frac{1}{a_d}\right)^{\frac{1}{\theta_d}}
      \frac{1}{\theta_d}\sum_{m=1}^N m^{\frac{1}{\theta_d}-1}
    \leq \left(\frac{1}{a_d}\right)^{\frac{1}{\theta_d}}
      \frac{1}{\theta_d}\int_{1}^{N}x^{\frac{1}{\theta_d}-1}\mathrm{d}x
      + h(1)
    \ll N^{\frac{1}{\theta_d}}\ll P,
  \end{align*}
  where we have used \eqref{eq:relation_P_N} in the last step. Thus, if
  $\abs{\alpha}\leq\frac{1}{N}$, then $P\ll
  N^{\frac{1}{\theta_d}}\ll\abs{\alpha}^{-\frac{1}{\theta_d}}$ and the lemma
  follows.
  
  On the other side, if $\frac{1}{N}<\abs{\alpha}\leq\frac12$, then we set
  $M=\floor{\abs{\alpha}^{-1}}$ such that
  \[
    M\leq \frac{1}{\abs{\alpha}}<M+1\leq N.
  \]
  We write $U(t)=\sum_{m\leq t} e(\alpha m)$, for short. By standard estimates
  for exponential sums (\textit{cf.} Lemma 4.7 in
  Nathanson~\cite{nathanson1996:additive_number_theory}), we obtain
  $U(t)\ll\norm{\alpha}^{-1}=\abs{\alpha}^{-1}$, where $\norm{\cdot}$ denotes
  the distance to the nearest integer. Thus by partial summation we obtain
  \begin{align*}
    \left(\frac{1}{a_d}\right)^{\frac{1}{\theta_d}}
      \frac{1}{\theta_d}\sum_{m=M+1}^N m^{\frac{1}{\theta_d}-1}e(\alpha m)
    &=h(N)U(N)-h(M)U(M)-\int_{M}^{N}U(t)h'(t)\mathrm{d}t\\
    &\ll \frac{M^{\frac{1}{\theta_d}-1}}{\abs{\alpha}}
      \ll \abs{\alpha}^{-\frac{1}{\theta_d}}
      \ll \min\left(P,\abs{\alpha}^{-\frac{1}{\theta_d}}\right)
  \end{align*}
  and the lemma also follows in this case.
\end{proof}

The first step replaces $F$ by $V$. Therefore we need the following lemma.
\begin{lem}\label{lem:comparing-F-and-V}
  If $\abs{\alpha}<\tau^{-1}$, then
  \[
    F(\alpha)-V(\alpha)
    \ll P^{\theta_{d-1}-\theta_d+1+v},
  \]
  where we have put $\theta_0=1$ if $d=1$.
\end{lem}

\begin{proof}
  \begin{align*}
    F(\alpha)-V(\alpha)
    =\sum_{m=1}^P e\left(\alpha\floor{f(m)}\right)-\left(\frac{1}{a_d}\right)^{\frac1{\theta_d}}\sum_{m=1}^N m^{\frac{1}{\theta_d}-1}e(\alpha m)
    =\sum_{m=1}^N u(m)e\left(\alpha m\right),
  \end{align*}
  where
  \begin{gather*}
      u(m)=c_m-\left(\frac{1}{a_d}\right)^{\frac{1}{\theta_d}}
      \frac{1}{\theta_d}m^{\frac{1}{\theta_d}-1}
  \end{gather*}
  and
  \begin{gather*}
    c_m=\#\{n\leq P\colon \floor{f(n)}=m\}.
  \end{gather*}

  We want to apply partial summation and therefore consider the summatory function
  \begin{align*}
    U(t)
    &:=\sum_{1\leq m\leq t}u(m)
    =\sum_{f(m)\leq t}c_m-\left(\frac{1}{a_d}\right)^{\frac{1}{\theta_d}}
    \frac{1}{\theta_d}\sum_{m\leq t}m^{\frac{1}{\theta_d}-1}\\
    &=\left(\frac{t}{a_d}\right)^{\frac{1}{\theta_d}}
    +\mathcal{O}\left(t^{\frac{1+\theta_{d-1}-\theta_d}{\theta_d}}\right)
    -\left(\frac{t}{a_d}\right)^{\frac{1}{\theta_d}}+\mathcal{O}(1)
    \ll \max\left(1,t^{\frac{1+\theta_{d-1}-\theta_d}{\theta_d}}\right).
  \end{align*}
  Thus by partial summation we obtain
  \begin{align*}
    F(\alpha)-V(\alpha)
    &=\sum_{m=1}^N u(m)e(\alpha m)=e(\alpha N)U(N)
      -2\pi i \alpha\int_{1}^{N}e(\alpha t)U(t)\mathrm{d}t\\
    &\ll U(N)+\abs{\alpha} N U(N) \ll P^{\theta_{d-1}-\theta_d+1+v}.
    \qedhere
  \end{align*}
\end{proof}

After comparing $F$ and $V$ we need to compare $F^s$ and $V^s$ as in
\eqref{eq:circle-decomposition}. Therefore we define $J^*(N)$ as
\begin{gather}\label{eq:pre-singular-integral}
  J^*(N):=\int_{-\tau^{-1}}^{+\tau^{-1}}V(\alpha)^s e\left(-\alpha N\right)\mathrm{d}\alpha.
\end{gather}
Then we get the following
\begin{prop}
  Let $s\geq2$ be an integer. Then
  \[
    \int_{-\tau^{-1}}^{+\tau^{-1}}F(\alpha)^s e\left(-\alpha
    N\right)\mathrm{d}\alpha
    =J^*(N)+\mathcal{O}\left(P^{s-\theta_d-\delta_1}\right)
  \]
  with $\delta_1>0$.
\end{prop}

\begin{proof}
  Since $F(\alpha)\ll P$ we get by Lemmas \ref{lem:estimate-of-V} and
  \ref{lem:comparing-F-and-V} that
  \begin{align*}
    F^s-V^s=(F-V)\left(F^{s-1}+F^{s-2}V+\cdots+V^{s-1}\right)
    \ll 
    P^{s+\theta_{d-1}-\theta_d+v}.
  \end{align*}
  Thus
  \begin{align*}
    \int_{-\tau^{-1}}^{+\tau^{-1}}\abs{F^s-V^s}\mathrm{d}\alpha
    \ll \tau^{-1} P^{s+\theta_{d-1}-\theta_d+v}
    \ll P^{s-\theta_d-\delta_1},
  \end{align*}
  where $\delta_1:=\theta_d-\theta_{d-1}-2v>0$.
\end{proof}

\section{The singular integral}\label{sec:the-singular-integral} In this section
we want to get rid of the dependency of $J^*(N)$ on $\tau$. Therefore we define
the singular integral $J$ by
\[
  J(N):=\int_{-\tfrac12}^{+\tfrac12} V(\alpha)^s e(-\alpha N)\mathrm{d}\alpha.
\]
Then we first show that $J(N)$ is small and $J^*(N)$ and $J(N)$ are close.
\begin{lem}\label{lem:estimates-of-J}
  If $s>\theta_d$, then
  \[
    J(N)\ll P^{s-\theta_d}
    \quad\text{and}\quad
    J^*(N)=J(N)+\mathcal{O}\left(P^{s-\theta_d-\delta_2}\right),
  \]
  where $\delta_2>0$.
\end{lem}

\begin{proof}
  By Lemma \ref{lem:estimate-of-V} we obtain
  \begin{align*}
    J(N)
    &\ll\int_{0}^{\tfrac{1}{2}}
      \min\left(P,\abs{\alpha}^{-\frac{1}{\theta_d}}\right)^s
      \mathrm{d}\alpha\\
    &\ll\int_{0}^{\frac1N}P^s\mathrm{d}\alpha
    +\int_{\frac1N}^{\frac12}\abs{\alpha}^{-\frac{s}{\theta_d}}\mathrm{d}\alpha\ll P^{s-\theta_d}.
  \end{align*}

  For the second part we also use Lemma \ref{lem:estimate-of-V} and get
  \begin{align*}
    J(N)-J^*(N)
    &=\int_{\tau^{-1}\leq\abs{\alpha}\leq\tfrac12}
      V(\alpha)^s e\left(-\alpha N\right)\mathrm{d}\alpha\\
    &\ll \int_{\tau^{-1}}^{\tfrac12}
      \alpha^{-\frac{s}{\theta_d}}\mathrm{d}\alpha
    \ll \tau^{\frac{s}{\theta_d}-1}\ll P^{s-\theta_d-\delta_2},
  \end{align*}
  where $\delta_2:=v\left(\tfrac{s}{\theta_d}-1\right)>0$.
\end{proof}

The following lemma allows us to obtain the main term of the singular integral
$J(N)$ as a product of Euler Gamma functions $\Gamma$.
\begin{lem}[\cite{nathanson1996:additive_number_theory}*{Lemma 5.3}]
  \label{lem:nat-lem5.3}
  Let $\alpha,\beta\in\RR$ such that $0<\beta<1$ and $\alpha\geq\beta$. Then 
  \[
    \sum_{m=1}^{N-1}m^{\beta-1}(N-m)^{\alpha-1}
    =N^{\alpha+\beta-1}\frac{\Gamma(\alpha)\Gamma(\beta)}{\Gamma(\alpha+\beta)}
      +\mathcal{O}\left(N^{\alpha-1}\right),
  \]
  where the implied constant depends only on $\beta$.
\end{lem}

\begin{prop}
  Let $s\geq2$ be an integer. Then we have
  \[
    J_s(N):=\int_{-\tfrac12}^{\tfrac12}V(\alpha)^se(-\alpha N)\mathrm{d}\alpha
    =\left(\frac{1}{a_d}\right)^{\frac{s}{\theta_d}}
    \frac{\Gamma\left(1+\frac{1}{\theta_d}\right)^s}
      {\Gamma\left(\frac{s+1}{\theta_d}\right)}
    N^{\frac{s}{\theta_d}-1}
    +\mathcal{O}\left(N^{\frac{s-1}{\theta_d}-1}\right).
  \]
\end{prop}

\begin{proof}
  By the definition of $V$ in \eqref{eq:V} we obtain that
  \begin{align*}
    J_s(N)
    &=\left(\frac{1}{a_d}\right)^{\frac{s}{\theta_d}}
    \left(\frac{1}{\theta_d}\right)^s
    \sum_{m_1\leq N}\cdots \sum_{m_s\leq N}
    (m_1\cdots m_s)^{\frac{1}{\theta_d}-1}
    \int_{-\tfrac12}^{\tfrac12}e\left(\left(m_1+\cdots+m_s\right)\alpha\right)\mathrm{d}\alpha\\
    &=\left(\frac{1}{a_d}\right)^{\frac{s}{\theta_d}}
    \left(\frac{1}{\theta_d}\right)^s
    \sum_{\substack{m_1+\cdots+m_s=N\\ 1\leq m_1,\ldots,m_s\leq N}}
    (m_1\cdots m_s)^{\frac{1}{\theta_d}-1}.
  \end{align*}

  Starting with the base case ($s=2$) we have by Lemma \ref{lem:nat-lem5.3}
  (with $\alpha=\beta=\tfrac{1}{\theta_d}$) that
  \begin{align*}
    J_2(N)
    &=\left(\frac{1}{a_d}\right)^{\frac{s}{\theta_d}}
    \left(\frac{1}{\theta_d}\right)^s
    \sum_{m=1}^{N}
    m^{\frac{1}{\theta_d}-1}\left(N-m\right)^{\frac{1}{\theta_d}-1}\\
    &\ll \frac{\Gamma\left(1+\frac{1}{\theta_d}\right)^2}
      {\Gamma\left(\frac{2}{\theta_d}\right)}
    N^{\frac{2}{\theta_d}-1}
    +\mathcal{O}\left(N^{\frac{1}{\theta_d}-1}\right).
  \end{align*}

  Now we suppose that the proposition holds for $s$ and we need to show it for
  $s+1$. Thus
  \begin{align*}
    J_{s+1}(N)
    &=\left(\frac{1}{a_d}\right)^{\frac{1}{\theta_d}}
    \left(\frac{1}{\theta_d}\right)
    \sum_{m=1}^{N}
    m^{\frac{1}{\theta_d}-1}
    \int_{-\tfrac12}^{\tfrac12}V(\alpha)^s e\left(\alpha(m-N)\right)\mathrm{d}\alpha\\
    &=\left(\frac{1}{a_d}\right)^{\frac{1}{\theta_d}}
    \left(\frac{1}{\theta_d}\right)
    \sum_{m=1}^{N}
    m^{\frac{1}{\theta_d}-1} J_s(N-m).
  \end{align*}
  Therefore, by the induction hypothesis we get
  \begin{align*}
    &J_{s+1}(N)\\
    &\quad=\left(\frac{1}{a_d}\right)^{\frac{s+1}{\theta_d}}
    \left(\frac{1}{\theta_d}\right)^{s+1}
    \frac{\Gamma\left(1+\frac{1}{\theta_d}\right)^s}
      {\Gamma\left(\frac{s}{\theta_d}\right)}
    \sum_{m=1}^{N}
    m^{\frac{1}{\theta_d}-1}
    (N-m)^{\frac{s}{\theta_d}-1}
    +\mathcal{O}\left(
          \sum_{m=1}^{N}
    m^{\frac{1}{\theta_d}-1}
    (N-m)^{\frac{s-1}{\theta_d}-1}\right)\\
    &\quad=\left(\frac{1}{a_d}\right)^{\frac{s+1}{\theta_d}}
    \left(\frac{1}{\theta_d}\right)^{s+1}
    \frac{\Gamma\left(1+\frac{1}{\theta_d}\right)^{s+1}}
      {\Gamma\left(\frac{s+1}{\theta_d}\right)}
    N^{\frac{s+1}{\theta_d}-1}
    +\mathcal{O}\left(N^{\frac{s}{\theta_d}-1}\right).\qedhere
  \end{align*}
\end{proof}

Putting everything together so far we have
\begin{gather}\label{eq:major_arc_end}
    \int_{-\tau^{-1}}^{\tau^{-1}}F(\alpha)^se\left(-\alpha N\right)
  =\left(\frac{1}{a_d}\right)^{\frac{s}{\theta_d}}
    \frac{\Gamma\left(1+\frac{1}{\theta_d}\right)^s}
      {\Gamma\left(\frac{s+1}{\theta_d}\right)}
    N^{\frac{s}{\theta_d}-1}
  +\mathcal{O}\left(N^{\frac{s}{\theta_d}-1-\delta}\right)
\end{gather}
with $\delta>0$, which terminates our treatment of the integral over the major
arc. 

\section{The floor function}\label{sec:the-floor-function}

In the following three sections we estimate the integral over the minor arc.
Starting with a treatment of the floor function in the present section, we
continue with an estimate of the occurring exponential sums in the following.
Finally in Section \ref{sec:the-minor-arc} we put everything together in order
to show that the integral over the minor arc only contributes to the error term.

Our first observation is that
\begin{gather*}
  \int_{\tau^{-1}}^{1-\tau^{-1}} F(\alpha)^se\left(-\alpha N\right)
  \ll \left(\sup_{\norm{\alpha}\geq\tau^{-1}} \abs{F(\alpha)}\right)^s.
\end{gather*}
Thus the estimate of the integral over the minor arc boils down to an estimate
of the exponential sum $F(\alpha)$ for $\norm{\alpha}\geq\tau^{-1}$.

Let $B\geq1$ be a positive integer, which we will choose in Section
\ref{sec:the-minor-arc} below. In order to get rid of the floor function we
divide the unit interval into $B$ equally sized intervals:
\[
  [0,1[=\bigcup_{b=0}^{B-1}\left[\frac{b}{B},\frac{b+1}{B}\right[
    =:\bigcup_{b=0}^{B-1}I_b,
\]
say.

We fix $m\leq P$ for the instant. Then there exists $0\leq b<B$ such that
$\left\{ f(m)\right\}\in I_b$, where $\{x\}=x-\floor{x}$ denotes the fractional
part of $x$. Otherwise said $\left\{f(m)\right\}$ is close to $\tfrac{b}{q}$ and
we may write
\[
  e\left(\alpha\floor{f(m)}\right)
  =e\left(\alpha f(m)-\frac{\alpha b}{q}\right)\left(1+\mathcal{O}\left(\frac1q\right)\right).
\]

Repeating this process for all $m\leq P$ and summing up we obtain
\begin{gather}\label{eq:minor_arc_step_1}
  \abs{F(\alpha)}=\abs{\sum_{m\leq P}e\left(\alpha \floor{f(m)}\right)}
  \leq \sum_{b=0}^{q-1}\abs{\sum_{m\leq P}e\left(\alpha f(m)\right)\chi_{I_b}\left(f(m)\right)}
    +\mathcal{O}\left(\frac{P}{q}\right),
\end{gather}
where $\chi_I$ is the indicator function of the interval $I$.

Up to now we have replaced the floor function by an indicator function. We use
the following lemma to approximate this indicator function by a trigonometric polynomial.
\begin{lem}[\cite{vaaler1985:some_extremal_functions}*{Theorem 19}]
  \label{lem:vaaler}
  Let $I$ be an interval and $\chi_I$ its indicator function. Furthermore let
  $H>0$ be an integer. Then there exist $a_H(h)$ and $C_h$ for $0<\abs{h}\leq H$
  with $\abs{a_H(h)}\leq 1$ and $\abs{C_h}\leq 1$ such that for
  \[
    \chi_{I,H}^*(x)=\abs{I}+\frac1\pi\sum_{0<\abs{h}\leq H}\frac{a_H(h)}{\abs{h}}e\left(hx\right)
  \]
  we have
  \[
    \abs{\chi_{I}(x)-\chi_{I,H}^*(x)}
    \leq \frac{1}{H+1}\sum_{0<\abs{h}\leq H} C_h \left(1-\frac{\abs{h}}{H+1}\right)e(hx).
  \]
\end{lem}

\begin{rem}
  The exact values of $a_H(h)$ and $C_h$ may be obtained from the proof of
  \cite{vaaler1985:some_extremal_functions}*{Theorem 19}. Since we only want to
  show that the integral over the minor arc only contributes to the error term,
  these upper bounds are sufficient for our needs.
\end{rem}

Let $H>0$ be an integer, which we will also choose in Section
\ref{sec:the-minor-arc} below. Using Lemma \ref{lem:vaaler} we may replace
$\chi_{I_b}$ for $0\leq b\leq q-1$ by a trigonometric polynomial
$\chi_{I_b,H}^*$ and get
\begin{multline}\label{eq:minor_arc_step_2}
  \abs{\sum_{m\leq P}e\left(\alpha f(m)\right)\chi_{I_b}\left(f(m)\right)}\\
  \leq \abs{\sum_{m\leq P}e\left(\alpha f(m)\right)\chi_{I_b,H}^*\left(f(m)\right)}
    + \abs{\sum_{m\leq P}\left(\chi_{I_b}\left(f(m)\right)-\chi_{I_b,H}^*\left(f(m)\right)\right)}.
\end{multline}

Again using the estimates for the coefficients in Lemma \ref{lem:vaaler} we
obtain
\begin{gather}\label{eq:floor-estimate_part_1}
    \abs{\sum_{m\leq P}e\left(\alpha f(m)\right)\chi_{I_b,H}^*\left(f(m)\right)}
    \leq \frac{1}{q}\abs{\sum_{m\leq P}e(\alpha f(m))}
      +\frac{1}{\pi}\sum_{0<\abs{h}\leq H}\frac{1}{\abs{h}}
        \abs{\sum_{m\leq P}e\left(\left(\alpha + h\right)f(m)\right)}
\end{gather}
and
\begin{gather}\label{eq:floor-estimate_part_2}
    \abs{\sum_{m\leq P}\left(\chi_{I_b}\left(f(m)\right)-\chi_{I_b,H}^*\left(f(m)\right)\right)}
    \leq \frac{1}{H+1}\sum_{\abs{h}\leq H}\left(1-\frac{\abs{h}}{H+1}\right)
      \abs{\sum_{m\leq P}e(hf(m))},
\end{gather}
respectively. In the following section we will develop estimates for the
occurring exponential sums.

\section{Exponential sum estimates}\label{sec:exponential-sum-estimates}

We may summarize the exponential sums in \eqref{eq:floor-estimate_part_1} and
\eqref{eq:floor-estimate_part_2} as all being of the following form
\begin{gather}\label{eq:exp-sum-template}
  \sum_{m\leq P}e\left(\beta f(m)\right)
\end{gather}
with $\beta=\alpha$, $\beta=\alpha+h$ and $\beta=h$, respectively. Since
$\alpha$ is on the minor arc and $h\neq0$, we have 
\begin{gather}\label{eq:lower_bound_coeff_minor_arc}
  \abs{\beta}>P^{\nu-\theta_d}.
\end{gather}

For the estimation we will distinguish several cases according to the size of
$\beta$ and whether $\theta_d\in\ZZ$ or not. Our first two estimates deal with
the case of very small $\beta$. On the one hand we have the following lemma
originally due to Kusmin and Landau.
\begin{lem}[\cite{graham_kolesnik1991:van_der_corputs}*{Theorem 2.1}]
  \label{lem:kusmin-landau}
  Let $g$ be a continuously differentiable function defined on an interval $I$.
  If $0<\lambda<g'(x)<1-\lambda$ and $g'$ is monotone on $I$, then
  \[
    \sum_{n\in I}e(g(n))\ll\lambda^{-1}.
  \]
\end{lem}
On the other we have an estimate due to van der Corput.
\begin{lem}[\cite{graham_kolesnik1991:van_der_corputs}*{Theorem 2.2}]
  \label{lem:van-der-Corput}
  Let $g$ be a twice continuously differentiable function defined on an interval
  $I$. If there are $\lambda>0$ and $\eta\geq1$ such that
  $\lambda<\abs{g''(x)}\leq \eta\lambda$, then
  \[
    \sum_{n\in I}e(g(n))\ll \abs{I}\eta\lambda^{\frac12}+\lambda^{-\frac12}.
  \]
\end{lem}

In the case of larger $\beta$ we need to adapt Vinogradov's method to our needs.
Details on this method can be found in Chapter 6 of
\cite{titchmarsh1986:theory_riemann_zeta} or Chapter 6 of
\cite{vinogradov2004:method_trigonometrical_sums}. We establish the adaption
using two tools. The first will help us estimate the number of overlaps we have.
\begin{lem}[\cite{titchmarsh1986:theory_riemann_zeta}*{Lemma 6.11}]
  \label{lem:number-of-small-fractional-part}
  Let $M$ and $N$ be integers with $N>1$ and let $\phi$ be a real function
  defined on $[M,M+N[$. Suppose that there exist $\delta>0$ and $c\geq1$ with
  $c\delta\leq\tfrac12$ such that
  \[
    \delta\leq \phi(n+1)-\phi(n)\leq c\delta
    \quad\left(M\leq n\leq M+N-2\right).
  \]
  Then for any $D>0$ the number of $n$ such that $\norm{\phi(n)}\leq D\delta$ is
  less than
  \[
    \left(Nc\delta+1\right)\left(2D+1\right).
  \]
\end{lem}

The second ingredient is an estimate for the Vinogradov integral, defined for
integers $N,s,k\geq1$ by
\[
  J_{s,k}(N):=\int_{[0,1]^d}\abs{\sum_{n=1}^Ne(\alpha_1n+\cdots +\alpha_kn^k)}
  \mathrm{d}\alpha_1\cdots\mathrm{d}\alpha_k.
\]
\begin{lem}[\cite{bourgain_demeter_guth2016:proof_main_conjecture}*{Theorem 1.1}]
  \label{lem:vinogradov-main-conjecture}
  For $s\geq1$ and $k,N\geq2$ we have
  \[
    J_{s,k}\ll N^{s+\varepsilon} + N^{2s-\frac{k(k+1)}{2}+\varepsilon},
  \]
  where the implied constant depends on $\varepsilon$ and $k$.
\end{lem}

\begin{rem}
  This estimate was called the main conjecture in Vinogradov's mean value
  theorem. The case $k=3$ has been proven by Wooley
  \cite{wooley2016:cubic_case_main}.
\end{rem}

Now our main tool is the following variant of
\cite{titchmarsh1986:theory_riemann_zeta}*{Lemma 6.12}.
\begin{prop}\label{prop:vinogradov_exp_sum_estimate}
  Let $g$ be a real function defined on $]P,P+Q]$ with $Q\leq P$. Suppose that
  for $k\geq2$ there exist $\lambda$ and $c$ such that
  \[
    \lambda\leq\frac{g^{(k+1)}(x)}{(k+1)!}\leq c\lambda
    \quad(P<x\leq P+Q).
  \]
  If there exists $\delta>0$ such that $Q^{-k-1+\delta}\leq \lambda\leq Q^{-1}$,
  then
  \[
    \abs{S}:=\abs{\sum_{n=P+1}^{P+Q} e\left(g(n)\right)} \ll Q^{1-\frac{\delta}{k(k+1)}}.
  \]
\end{prop}

\begin{proof}
  Let $q\geq1$ be defined by
  \[
    q=1+\floor{\lambda^{-\frac{1}{k+1}}}.
  \]
  Furthermore for $P+1\leq n\leq P+Q-q$ we set
  \begin{gather}\label{eq:T(n)}
    T(n)=\sum_{m=1}^q e\left(g(n+m)-g(n)\right).
  \end{gather}

  Then
  \begin{equation}\label{eq:S_correlation}
    \begin{split}
    q\abs{S}
    &=\abs{\sum_{m=1}^q\sum_{n=P+1}^{P+Q}e\left(g(n)\right)}
    \leq\abs{\sum_{m=1}^q\sum_{n=P+1+m}^{P+Q+m}e\left(g(n)\right)}+\sum_{m=1}^q q\\
    &=\abs{\sum_{m=1}^q\sum_{n=P+1}^{P+Q}e\left(g(m+n)\right)}+q^2
    =\abs{\sum_{n=P+1}^{P+Q}\sum_{m=1}^q e\left(g(m+n)\right)}+q^2\\
    &\leq\sum_{n=P+1}^{P+Q}\abs{\sum_{m=1}^q e\left(g(m+n)\right)}+q^2\\
    &\leq Q^{1-\frac{1}{2\ell}}\left\{\sum_{n=P+1}^{P+Q-q}\abs{T(n)}^{2\ell}\right\}^{\frac{1}{2\ell}}+q^2,
    \end{split}
  \end{equation}
  where we used Hölder's inequality with an integer $\ell\geq1$, which we will
  choose later.

  Now we concentrate on $T(n)$. To this end, we write
  \[
    A_j=A_j(n)=\frac{g^{(j)}(n)}{j!}
    \quad(1\leq j\leq k),
  \]
  for short, and define the domain $\Omega(n)\subset\RR^k$ by
  \[
    \Omega(n):=\left\{
      (t_1,\ldots,t_k)\in \RR^k\colon
      \abs{t_j-A_j}\leq\tfrac12 q^{-j}\text{ for }
      1\leq j\leq k
    \right\}.
  \]
  Furthermore for $1\leq m\leq q$ we set
  \[
    \Delta(m)=g(m+n)-g(n)-\left(t_k m^k+\cdots+t_1m\right).
  \]
  Then, by partial summation,
  \begin{align*}
    T(n)
    &=\sum_{m=1}^q e\left(g(m+n)-g(n)\right)
    =\sum_{m=1}^q e\left(\Delta(m)\right)e\left(t_k m^k+\cdots+t_1m\right)\\
    &=S(q)e\left(\Delta(q)\right)-2\pi i\int_0^q S(y)\Delta'(y)e\left(\Delta(y)\right)\mathrm{d}y,
  \end{align*}
  where
  \begin{gather}\label{eq:S(q)}
    S(q)=\sum_{m=1}^q e\left(t_km^k+\cdots+t_1m\right).
  \end{gather}

  Using Taylor expansion we have
  \[
    \Delta'(y)=\sum_{j=1}^k j\left(\frac{g^{(j)}(n)}{j!}-t_j\right)y^{j-1}
    +\mathcal{O}\left(g^{(k+1)}(n+\xi y)y^k\right)
  \]
  with $0<\xi<1$. Thus for $(t_1,\ldots,t_k)\in\Omega(n)$ we have
  \[
    \abs{\Delta'(y)}
    \ll\sum_{j=1}^k j\frac{1}{2}q^{-j}q^{j-1}+\lambda q^{k}
    \ll q^{-1}
  \]
  and therefore
  \[
    \abs{T(n)}\ll \abs{S(q)}+\frac1q\int_0^q\abs{S(y)}\mathrm{d}y.
  \]
  Taking to the power $2\ell$ and averaging over $\Omega(n)$ yields
  \[
    \abs{T(n)}^{2\ell}\ll q^{\frac{k(k+1)}{2}}\int_{\Omega(n)}
    \left(\abs{S(q)}+\frac1q\int_0^q\abs{S(y)}\mathrm{d}y\right)^{2\ell}
    \mathrm{d}t_1\cdots \mathrm{d}t_k.
  \]

  Now we want to sum over $n$ and need to consider the possible overlaps of
  $\Omega(n)$ and $\Omega(n')$ for $n\neq n'$. To this end we fix an integer
  $n'\in[P+1,P+Q-q]$. Then a necessary condition for an overlap is that
  \[
    \norm{A_k(n)-A_k(n')}\leq q^{-k}\leq \lambda q.
  \]
  We set $\phi(n)=A_k(n)-A_k(n')$, $c=2$, $\delta=\lambda(k+1)$ and $D=q/(k+1)$
  and obtain by an application of Lemma
  \ref{lem:number-of-small-fractional-part} that for given $n'$ there are no
  more than
  \[
    \left(2Q\lambda(k+1)+1\right)\left(\frac{2q}{k+1}+1\right)
    \leq\left(2k+3\right)\left(\frac{2q}{k+1}+1\right)\leq 3kq
  \]
  possible overlaps. This is independent of our choice of $n'$ and therefore we
  get for the sum over the $n$ that
  \begin{multline*}
    \sum_{n=P+1}^{P+Q-q}\int_{\Omega(n)}
    \left(\abs{S(q)}+\frac1q\int_0^q\abs{S(y)}\mathrm{d}y\right)^{2\ell}
    \mathrm{d}t_1\cdots \mathrm{d}t_k\\
    \leq 3kq \int_{[0,1]^k}\left(\abs{S(q)}+\frac1q\int_0^q\abs{S(y)}\mathrm{d}y\right)^{2\ell}
    \mathrm{d}t_1\cdots \mathrm{d}t_k.
  \end{multline*}
  Finally, since
  \[
    \left(\abs{S(q)}+\frac1q\int_0^q\abs{S(y)}\mathrm{d}y\right)^{2\ell}
    \leq 2^{2\ell -1}\left(\abs{S(q)}^{2\ell}+\frac1q\int_0^q\abs{S(y)}^{2\ell}\mathrm{d}y\right)
  \]
  we get that
  \[
    \int_{[0,1]^k}\left(\abs{S(q)}+\frac1q\int_0^q\abs{S(y)}\mathrm{d}y\right)^{2\ell}
    \mathrm{d}t_1\cdots \mathrm{d}t_k
    \leq J_{\ell,k}(q).
  \]

  Choosing $\ell=k(k+1)/2$ an application of Lemma
  \ref{lem:vinogradov-main-conjecture} yields
  \[
    J_{\ell,k}(q)\ll q^{\frac{k(k+1)}{2}+\varepsilon}.
  \]
  
  Plugging everything in we obtain
  \begin{gather*}
    \abs{S}
    \ll Q^{1-\frac{1}{2\ell}} q^{-1}\left\{
      q^{\frac{k(k+1)}{2}}J_{\ell,k}(q)
      \right\}^{\frac{1}{2\ell}}+q
    \ll Q^{1-\frac{\delta}{k(k+1)}}\qedhere
  \end{gather*}
\end{proof}

Now we have all the tools in hand in order to successfully estimate the
exponential sum in \eqref{eq:exp-sum-template}. For our treatment we first split
the summation over $m$ into $W$ intervals of the form $]P/2^{w},P/2^{w-1}]$:
\begin{gather}\label{eq:dyadic_split}
  \abs{\sum_{m\leq P} e\left(\beta f(m)\right)}
  \leq\abs{\sum_{m\leq P/2^{W}} e\left(\beta f(m)\right)}
    +\sum_{0\leq w<W}\abs{\sum_{P/2^{w+1}<m\leq P/2^{w}} e\left(\beta f(m)\right)}.  
\end{gather}
In the following we fix $0\leq w<W$ and estimate the sum
\begin{gather*}
  S_w=\sum_{m=Q+1}^{2Q}e\left(\beta f(m)\right)
\end{gather*}
with $Q=P2^{-w-1}$.

For the rest of this section we follow the proof of Nakai and Shiokawa
\cite{nakai_shiokawa1990:class_normal_numbers} for the estimation of a related
exponential sum. To this end let $\theta$ be the largest $\theta_j\not\in\ZZ$.
Then we first distinguish the case $\theta=\theta_d$, meaning that the
``leading'' exponent is not an integer, and $\theta\neq\theta_d$.
\begin{itemize}
  \item \textbf{Case 1:} If $\theta_d = \theta$ (the ``leading'' exponent is not
  an integer), then we consider two cases according to
  the size of $\beta$.

  \begin{itemize}
    \item \textbf{Case 1.1:} If $\norm{\beta}\geq Q^{-\theta_d+\nu}$, then we
    set $k=\left\lceil \theta_d\right\rceil\geq2$. Since $\theta_d\not\in
    \ZZ$, we have
    \[
      f^{(k+1)}(x)\sim a_d\theta_d(\theta_d-1)\cdots(\theta_d-k)x^{\theta_d-k-1}.
    \]
    Thus 
    \[
      \lambda\leq \frac{\beta f^{(k+1)}(x)}{(k+1)!}\leq c\lambda
      \quad(Q<x\leq 2Q)
    \]
    with
    \[
      \lambda=c\beta Q^{\theta_d-(k+1)},
    \]
    where $c$ is a constant depending only on $f$.

    Thus
    \[
      Q^{\nu-k-1}\leq \lambda\asymp \leq Q^{\theta_d-k-1}\leq Q^{-1}.
    \]
    An application of Proposition \ref{prop:vinogradov_exp_sum_estimate} yields
    \begin{gather}\label{eq:S_w_case_1.1}
      S_w\ll Q^{1-\frac{\nu}{k(k+1)}}
    \end{gather}
    
    \item \textbf{Case 1.2:} By \eqref{eq:lower_bound_coeff_minor_arc} we have
    $P^{-\theta_d+\nu}\leq \norm{\beta}< Q^{-\theta_d+\nu}$. Thus
    \[
      \beta f'(x)\geq P^{-\theta_d+\nu}Q^{\theta_d-1}
    \]
    and an application of Lemma \ref{lem:kusmin-landau} yields
    \begin{gather}\label{eq:S_w_case_1.2}
      S_w\ll Q^{1-\theta_d}P^{\theta_d-\nu}.
    \end{gather}
  \end{itemize}
  Since $Q=P2^{-w-1}$ decreases with $w$ there exists $0\leq W'\leq W$ such that
  for $0\leq w<W'$ we have \eqref{eq:S_w_case_1.1} and for $W'\leq w<W'$
  \eqref{eq:S_w_case_1.2} holds. Plugging this into \eqref{eq:dyadic_split}
  yields
  \begin{equation}\label{eq:exp_sum_non_integer}
    \begin{split}
      \abs{\sum_{m\leq P}e\left(\beta f(m)\right)}
      &\leq \sum_{w=0}^{W'-1}\left(P2^{-w-1}\right)^{1-\frac{\nu}{k(k+1)}}
        + \sum_{w=W'}^{W-1}\left(P2^{-w-1}\right)^{1-\theta_d}P^{\theta_d-\nu}
        + P2^{-W}\\
      &\ll P^{1-\frac{\nu}{k(k+1)}}+P^{1-\nu}2^{(\theta_d-1)W}+P2^{-W}\\
      &\ll P^{1-\frac{\nu}{\left\lceil \theta_d\right\rceil \left(\left\lceil \theta_d\right\rceil+1\right)}}    
    \end{split}
  \end{equation}
  for a suitable choice of $W$.

\item \textbf{Case 2:} $\theta_d\neq\theta$. Note that in this case there exists
  a $1\leq h<d$ such that $\theta=\theta_h$ and
  $\theta_{h+1},\ldots,\theta_d\in\ZZ$.
  
  This time we consider four different cases according to the size of
  $\norm{\beta}$.
  \begin{itemize}
  \item \textbf{Case 2.1:} $\norm{\beta}>Q^{-\theta_h+\rho}$. This case is analog
    to \textbf{Case 1.1}. Let $k=\theta_d\geq2$. Then
    \[
      f^{(k+1)}(x)\sim a_h \theta(\theta-1)\cdots(\theta-k)x^{\theta-k-1}
    \]
    and we have
    \[
      \lambda\leq\frac{\beta f^{(k+1)}(x)}{(k+1)!}\leq c\lambda
      \quad(Q<x\leq 2Q)
    \]
    with
    \[
      \lambda=c\beta Q^{\theta-k-1},
    \]
    where $c$ is a constant depending on $f$ only. Since
    $\norm{\beta}>Q^{-\theta_h+\rho}$ in this case we obtain
    \[
      Q^{\rho-k-1}\leq \lambda\leq Q^{-1}
    \]
    and an application of Proposition \ref{prop:vinogradov_exp_sum_estimate}
    yields
    \begin{gather}\label{eq:S_w_case_2.1}
      S_w\ll Q^{1-\frac{\rho}{\theta_d(\theta_d+1)^2}}.
    \end{gather}
    
    \item \textbf{Case 2.2:} $\norm{\beta}\leq P^{-\theta_d+1-\rho}$. Using
    \eqref{eq:lower_bound_coeff_minor_arc} we have
    \[
      P^{-\theta_d+\nu}Q^{\theta_d-1}\ll \beta f'(x) \ll P^{-\rho}
      \quad\left(Q<x\leq 2Q\right)
    \]
    Thus, by Lemma \ref{lem:kusmin-landau} we obtain
    \begin{gather}\label{eq:S_w_case_2.2}
      S_w\ll P^{\theta_d-\nu}Q^{1-\theta_d}.
    \end{gather}

    \item \textbf{Case 2.3:} $P^{-\theta_d+1-\rho}\leq
    \norm{\beta}<P^{-\theta_d+2-\rho}$. This implies that $\theta_d\geq2$. Then we have
    \[
      \lambda
      \ll\abs{\beta}f''(x)
      \ll
      \eta\lambda
    \]
    with
    \[
      \lambda = P^{-\theta_d+1-\rho_2}Q^{\theta_d-2}
      \quad\text{and}\quad
      \eta = P.
    \]
    Thus an application of Lemma \ref{lem:van-der-Corput} yields
    \begin{equation}\label{eq:S_w_case_2.3}
      \begin{split}
        S_w
        &\ll P^{\frac{-\theta_d+3-\rho_2}{2}}Q^{\theta_d/2-1}
          +P^{\frac{\theta_d-1+\rho_2}{2}}Q^{1-\theta_d/2}.
      \end{split}
    \end{equation}

    \item \textbf{Case 2.4:} $P^{-\theta_d+2-\rho}<\norm{\beta}\leq
    Q^{-\theta_h+\rho}$. Since $\theta_d\in\ZZ$, this implies that
    $\theta_d\geq3$ because otherwise this range would be empty for sufficiently
    large $P$.
    
    The idea is to adapt the proof of Proposition
    \ref{prop:vinogradov_exp_sum_estimate} similar to the last case of Nakai and
    Shiokawa \cite{nakai_shiokawa1990:class_normal_numbers}. Therefore we set
    $k=\theta_d-1$ and
    \[
      q=\left\lfloor \beta^{-\frac{1}{k}}\right\rfloor.
    \]
    
    Defining $T(n)$ in the same way as in \eqref{eq:T(n)} we obtain following
    the lines of \eqref{eq:S_correlation} that
    \[
      q\abs{S_w} \leq Q^{1-\frac{1}{2\ell}}
      \left\{\sum_{n=Q+1}^{2Q-q}\abs{T(n)}^{2\ell}\right\}^{\frac{1}{2\ell}}+q^2,
    \]
    where $\ell\geq1$ is an integer.

    Now we use that $\theta_d$ is an integer and for $0\leq y\leq q$ and
    $n\in\ZZ$ we set 
    \[
      \Delta(y)=\beta f(n+y)-\beta f(n)-\left(t_1y+t_2y^2+\cdots+t_ky^k+\beta a_dy^{\theta_d}\right).
    \]
    Then by partial summation we have
    \begin{align*}
      T(n)
      &=\sum_{m=1}^q e\left(\Delta(m)\right) e\left(t_1m+t_2m^2+\cdot+t_km^k\right)\\
      &=S(q)e\left(\Delta(q)\right)-2\pi i\int_{0}^{q}S(y)\Delta'(y)e\left(\Delta(y)\right)\mathrm{d}y,
    \end{align*}
    where $S(y)$ is as in \eqref{eq:S(q)}. Using Taylor expansion we obtain for
    $\Delta'(y)$
    \[
      \Delta'(y)
      =\sum_{j=1}^k j\left(\frac{\beta f^{(j)}(n)}{h!}-t_j\right)y^{j-1}
        +\mathcal{O}\left(\beta\left(f^{(k+1)}(n+\xi y)-a_d(k+1)!\right)y^{k+1}\right)
    \]
    with $0\leq \xi\leq 1$. Similar to the proof of Proposition
    \ref{prop:vinogradov_exp_sum_estimate} we set for $Q<n\leq 2Q-q$
    \[
      \Omega(n):=\left\{
      (t_1,\ldots,t_k)\in \RR^k\colon
      \abs{\frac{\beta f^{(j)}(n)}{j!}-t_j}\leq\tfrac12 q^{-j}\text{ for }
      1\leq j\leq k
    \right\}.
    \]
    Thus for $(t_1,\ldots,t_k)\in\Omega(n)$ we have
    \[
      \Delta'(y)\ll q^{-1}+Q^{\theta_h-\theta_d}\abs{\beta}q^k
      \ll q^{-1}+Q^{\theta_h-\theta_d}\ll q^{-1}.
    \]

    Following the rest of the proof of Proposition
    \ref{prop:vinogradov_exp_sum_estimate} we obtain
    \begin{gather}\label{eq:S_w_case_2.4}
      S_w\ll Q^{1-\frac{\rho}{(\theta_d-1)\theta_d^2}}.
    \end{gather}
  \end{itemize}

  Again since $Q=P2^{-w}$ decreases with $w$ we define $W'$, $W''$ and $W'''$
  such that for $0\leq w<W'$ we have \textbf{Case 2.1}, for $W'\leq w<W''$ we
  have \textbf{Case 2.4}, for $W''\leq w<W'''$ we have \textbf{Case 2.3} and for
  $W'''\leq w<W$ we have \textbf{Case 2.2}. Plugging everything into
  \eqref{eq:dyadic_split} we obtain
  \begin{equation}\label{eq:exp_sum_with_integer}
    \begin{split}
      \abs{\sum_{m\leq P}e\left(\beta f(m)\right)}
      &\leq \sum_{w=0}^{W'-1}\left(P2^{-w-1}\right)^{1-\frac{\rho}{\theta_d(\theta_d+1)^2}}
        + \sum_{w=W'}^{W''-1}\left(P2^{-w-1}\right)^{1-\frac{\rho}{(\theta_d-1)\theta_d^2}}\\
      &\quad + \sum_{w=W''}^{W'''-1}
        \left(P^{\frac{-\theta_d+3-\rho_2}{2}}\left(P2^{-w-1}\right)^{\theta_d/2-1}
        + P^{\frac{\theta_d-1+\rho_2}{2}}\left(P2^{-w-1}\right)^{1-\theta_d/2}\right)\\
      &\quad + \sum_{w=W'''}^{W-1}
        P^{\theta_d-\nu}\left(P2^{-w-1}\right)^{1-\theta_d}
        + P2^{-W}\\
      &\ll P^{1-\frac{\nu}{\theta_d(\theta_d+1)^2}}+P^{\frac{1-\rho_2}{2}}2^{\left(\theta_d/2-1\right)W}+P^{1-\nu}2^{(\theta_d-1)W}+P2^{-W}\\
      &\ll P^{1-\frac{\nu}{\theta_d(\theta_d+1)^2}},
    \end{split}
  \end{equation}
  again by a suitable choice of $W$.
\end{itemize}

We may combine \eqref{eq:exp_sum_non_integer} and
\eqref{eq:exp_sum_with_integer} in order to obtain
\begin{gather}\label{eq:exp_sum_estimate}
  \sum_{m\leq P} e(\beta f(m)) \ll P^{1-\frac{\nu}{\lceil\theta_d\rceil(\lceil\theta_d\rceil+1)}}.
\end{gather}

\section{The minor arc}\label{sec:the-minor-arc}

Now we have all tools at hand for estimating the minor arc.
Starting with
\eqref{eq:floor-estimate_part_1} we get by our exponential sum estimate
\eqref{eq:exp_sum_estimate} that
\[
  \abs{\sum_{m\leq P}e\left(\alpha f(m)\right)\chi_{I_b,H}^*\left(f(m)\right)}
    \ll \frac{1}{q} P^{1-\frac{\nu}{\left\lceil \theta_d\right\rceil \left(\left\lceil \theta_d\right\rceil+1\right)}}
      +(\log H)P^{1-\frac{\nu}{\left\lceil \theta_d\right\rceil \left(\left\lceil \theta_d\right\rceil+1\right)}}.
\]
Similarly we get for \eqref{eq:floor-estimate_part_2} that
\[
    \abs{\sum_{m\leq P}\chi_{I_b}\left(f(m)\right)-\chi_{I_b,H}^*\left(f(m)\right)}
    \ll P^{1-\frac{\nu}{\left\lceil \theta_d\right\rceil \left(\left\lceil \theta_d\right\rceil+1\right)}}.
\]

Plugging these two estimates into \eqref{eq:minor_arc_step_2} yields
\[
  \abs{\sum_{m\leq P}e\left(\alpha f(m)\right)\chi_{I_b}\left(f(m)\right)}
  \ll \left(q^{-1}+\log H\right)
    P^{1-\frac{\nu}{\left\lceil \theta_d\right\rceil \left(\left\lceil \theta_d\right\rceil+1\right)}}.
\]

Finally we put this estimate into \eqref{eq:minor_arc_step_1} and obtain
\[
  \abs{F(\alpha)}
  \ll \left(1+q\log H\right)
      P^{1-\frac{\nu}{\left\lceil \theta_d\right\rceil \left(\left\lceil \theta_d\right\rceil+1\right)}}
    +\frac{P}{q}
  \ll P^{1-\frac{\nu}{2\left\lceil \theta_d\right\rceil \left(\left\lceil \theta_d\right\rceil+1\right)}+\varepsilon},
\]
where we have chosen
\[
  q = P^{\frac{\nu}{2\left\lceil \theta_d\right\rceil \left(\left\lceil \theta_d\right\rceil+1\right)}}
\]

Putting everything together we obtain for the minor arc.
\begin{gather}\label{eq:minor_arc_end}
  \int_{\tau^{-1}}^{1-\tau^{-1}} F(\alpha)^se\left(-\alpha N\right)
  \ll P^{s-\frac{s\nu}{2\lceil\theta_d\rceil(\lceil\theta_d\rceil+1)^2}}.
\end{gather}

\section{Proof of Theorem \ref{thm:main}}\label{sec:conclusion}

Using \eqref{eq:circle-decomposition} with \eqref{eq:major_arc_end} and
\eqref{eq:minor_arc_end} yields
\begin{align*}
  r_{f,s}(N)
  &=\int_{-\tau^{-1}}^{+\tau^{-1}}F(\alpha)^s e\left(-\alpha N\right)\mathrm{d}\alpha
    + \int_{\tau^{-1}}^{1-\tau^{-1}}F(\alpha)^s e\left(-\alpha N\right)\mathrm{d}\alpha\\
  &=\left(\frac{1}{a_d}\right)^{\frac{s}{\theta_d}}
    \frac{\Gamma\left(1+\frac{1}{\theta_d}\right)^s}
      {\Gamma\left(\frac{s+1}{\theta_d}\right)}
    N^{\frac{s}{\theta_d}-1}
  + \mathcal{O}\left(N^{\frac{s}{\theta_d}-1-\delta}\right)
  + \mathcal{O}\left(N^{\frac{s}{\theta_d}-\frac{s\nu}{2\left\lceil \theta_d\right\rceil^2 \left(\left\lceil \theta_d\right\rceil+1\right)}+\varepsilon}\right).
\end{align*}

For
\[
  s>\frac{2}{\nu}\left\lceil \theta_d\right\rceil^2 \left(\left\lceil \theta_d\right\rceil+1\right)
\]
we have
\[
  \frac{s\nu}{2\left\lceil \theta_d\right\rceil^2 \left(\left\lceil \theta_d\right\rceil+1\right)}-\varepsilon>1+\delta
\]
proving the theorem.



\end{document}